\documentclass[12pt]{amsart}
\usepackage{a4wide}
\usepackage{amsmath}
\usepackage{amsfonts}
\usepackage{amssymb}
\usepackage{graphicx}
\usepackage{amsthm,graphicx,color}
\usepackage[labelfont=bf]{caption}
\usepackage{epstopdf}%
\usepackage{pifont}

\usepackage[colorlinks=true]{hyperref}
\hypersetup{linkcolor=red,citecolor=blue,filecolor=dullmagenta,urlcolor=magenta} % coloured links
\newcommand{\R} {\mathbb R}
\newcommand{\cuad}{{\sqcap\kern-.68em\sqcup}}

\newcommand{\ve}{\varepsilon}

\newcommand{\be}{\begin{equation}}
\newcommand{\ee}{\end{equation}}

\definecolor{darkgreen}{rgb}{0.2,0.7,0.1}

\newcommand{\re}{\operatorname{Re}}

\newcommand{\Com}{\mathbb{C}}

\newcommand{\al}{\alpha}
\newcommand{\bt}{\beta}
\newcommand{\ga}{\gamma}

\def\bm{\left( \begin{array}{cc}}
\def\endm{\end{array}\right)}
\providecommand{\norm}[1]{\left\| #1 \right\|}

\newcommand{\ba}{\begin{equation*}}
\newcommand{\ea}{\begin{equation*}}
\newcommand{\bea}{\begin{eqnarray}}
\newcommand{\eea}{\end{eqnarray}}
\newcommand{\bee}{\begin{eqnarray*}}
\newcommand{\eee}{\end{eqnarray*}}
\newcommand{\ben}{\begin{enumerate}}
\newcommand{\een}{\end{enumerate}}

\numberwithin{equation}{section}

\newtheorem{thm}{Theorem}[section]
\newtheorem*{theorem*}{Theorem}
\newtheorem{prop}{Proposition}[section]

\theoremstyle{remark}
\newtheorem{rem}{Remark}[section]
\title[The Akhmediev breather is unstable]{The Akhmediev breather is unstable}

\author{Miguel A. Alejo}
\address{Universidade Federal de Santa Catarina, Brasil}
\email{miguel.alejo@ufsc.br}
\thanks{M. A. A. also would like to thank to the Departamento de Ingenier\'ia Matem\'atica (DIM) of U. Chile,
where part of this work was completed, for its kind hospitality and support. Funded by  Product. CNPq grant no. 305205/2016-1 and MathAmSud/Capes EEQUADD collaboration Math16-01.} 

\author{Luca Fanelli}
\address{Dipartimento di Matematica, SAPIENZA Universit\`a di Roma, P. le Aldo Moro 5, 00185 Roma}
\email{fanelli@mat.uniroma1.it}
\author{Claudio Mu\~noz}
\address{CNRS and Departamento de Ingenier\'{\i}a Matem\'atica and Centro
de Modelamiento Matem\'atico (UMI 2807 CNRS), Universidad de Chile, Casilla
170 Correo 3, Santiago, Chile.}
\email{cmunoz@dim.uchile.cl}
\thanks{C. M.  was partly funded by Chilean research grants FONDECYT 1150202, and CMM Conicyt PIA AFB170001}

\thanks{The authors were participants of the \emph{Third Workshop on Nonlinear Dispersive Equations}, held in Campinas, Brazil, during November 8-10 2017. They acknowledge the support and charming ambiance ensured by UNICAMP and the organizers of this event}


\begin{document}

%%%%%%%%%%%%%%%%%%%%%%%%%%%%%%%%%%%%%%%%%%%%%%%%%%%%%%%%%%%%%%%%%%%%%%%%%%%%%%%%%%%%%%%%%%%%%%%%%%
\begin{abstract}
In this note, we give a rigorous proof that the NLS periodic Akhmediev breather is unstable. The proof follows the ideas in \cite{Munoz1}, in the sense that a suitable modification of the Stokes wave is the global attractor of the local Akhmediev dynamics for sufficiently large time, and therefore the latter cannot be stable in any suitable finite energy periodic Sobolev space.
\end{abstract}

\maketitle

%\thanks{THIS WORK}
%\tableofcontents
\section{Introduction}

%\subsection{Setting of the problem}
 
Let $a\in (0,\frac12)$. The Akhmediev breather \cite{Akhmediev}
\be\label{Ak}
\begin{aligned}
A(t,x):=& ~ e^{it}\Bigg[  1+ \frac{\alpha^2 \cosh( \beta t) +i\beta \sinh(\beta t) }{ \sqrt{2a} \cos(\alpha x) -\cosh(\beta t)}\Bigg], \qquad t,x\in\R,\\
 \beta=&~ (8a(1-2a))^{1/2}, \quad \alpha=(2(1-2a))^{1/2},
\end{aligned}
\ee
 is a {\bf $\frac{2\pi}{a}$-periodic in space, localized in time} smooth solution to the \emph{focusing} cubic nonlinear Schr\"odinger equation (NLS) in one dimension:
\be\label{NLS}
i\partial_t u + \partial_x^2 u +|u|^{2} u=0, \quad u=u(t,x)\in \Com, \quad t,x\in \R.
\ee
See Fig.\ref{Fig:1}-\ref{Fig:2} for details. This equation appears as a model of propagation of light in nonlinear optical fibers (with different meanings for time and space variables), as well as in small-amplitude gravity waves on the surface of deep inviscid water. Additionally, this equation is completely integrable, as showed by Zakharov and Shabat \cite{Zakharov0}.

% \begin{figure}[h!] 
%     \includegraphics[scale=0.65]{Akhmediev_Fig.pdf}%width=0.25\linewidth]{ss01.pdf} 
%   \caption{Absolute value of the Akhmediev  breather \eqref{Ak}, with $a=0.25$. Note that $A$ is periodic in space, localized in time.}
%   \label{Fig:1} 
% \end{figure}

\begin{minipage}{\linewidth}
      \centering
      \begin{minipage}{0.45\linewidth}
        %  \begin{figure}[b]{0.5\linewidth}
              \includegraphics[scale=0.317]{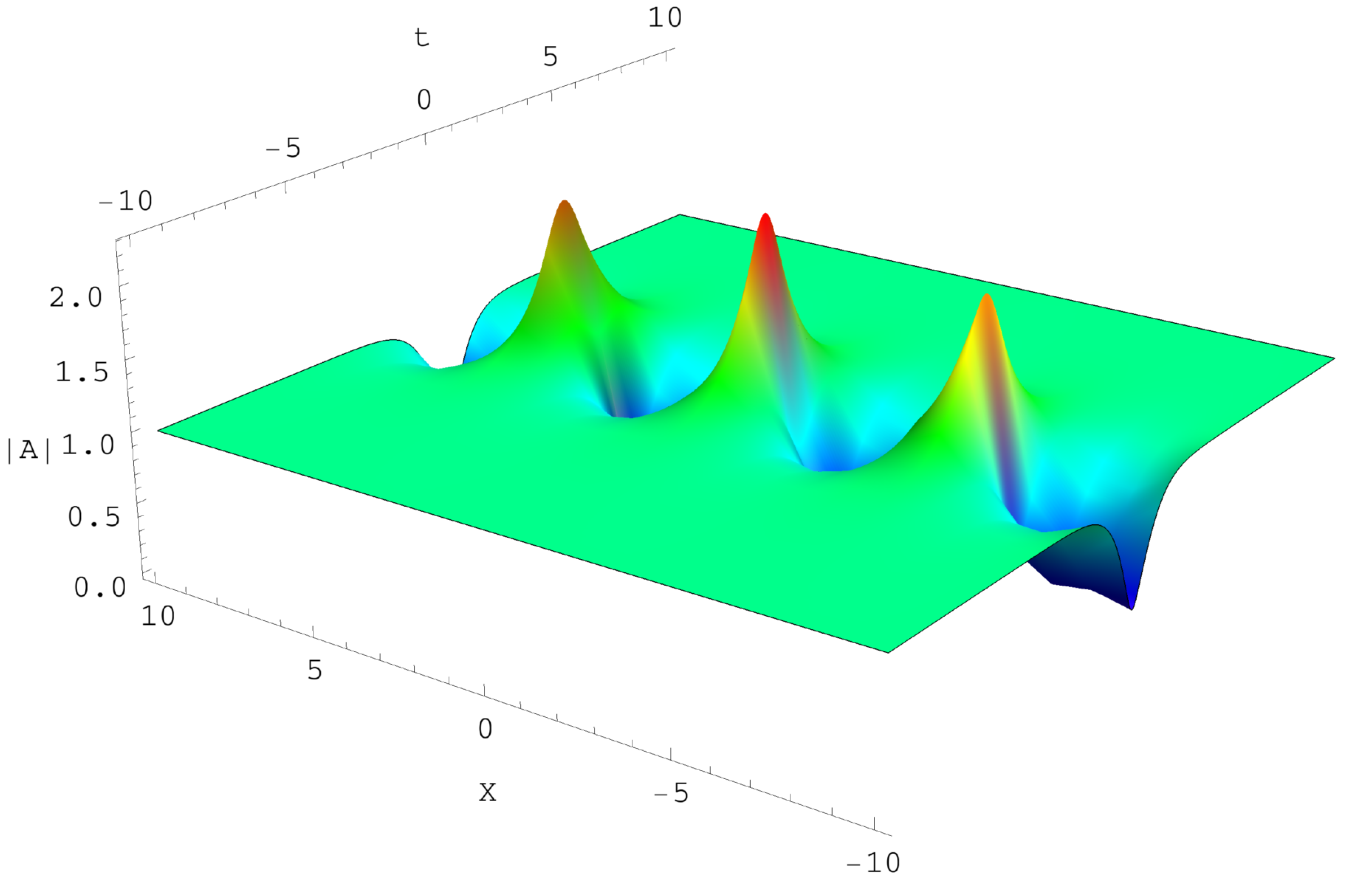}%[width=\linewidth]{A1.pdf}
              %\caption{{\small $|A|$ with $a=0.2$.}}
                  \captionof{figure}{\small $|A|$ with $a=0.2$.}\label{Fig:1} 
        %  \end{figure}
      \end{minipage}
      \hspace{0.05\linewidth}
      \begin{minipage}{0.45\linewidth}
        %  \begin{figure}[b]{0.5\linewidth}
              \includegraphics[width=\linewidth]{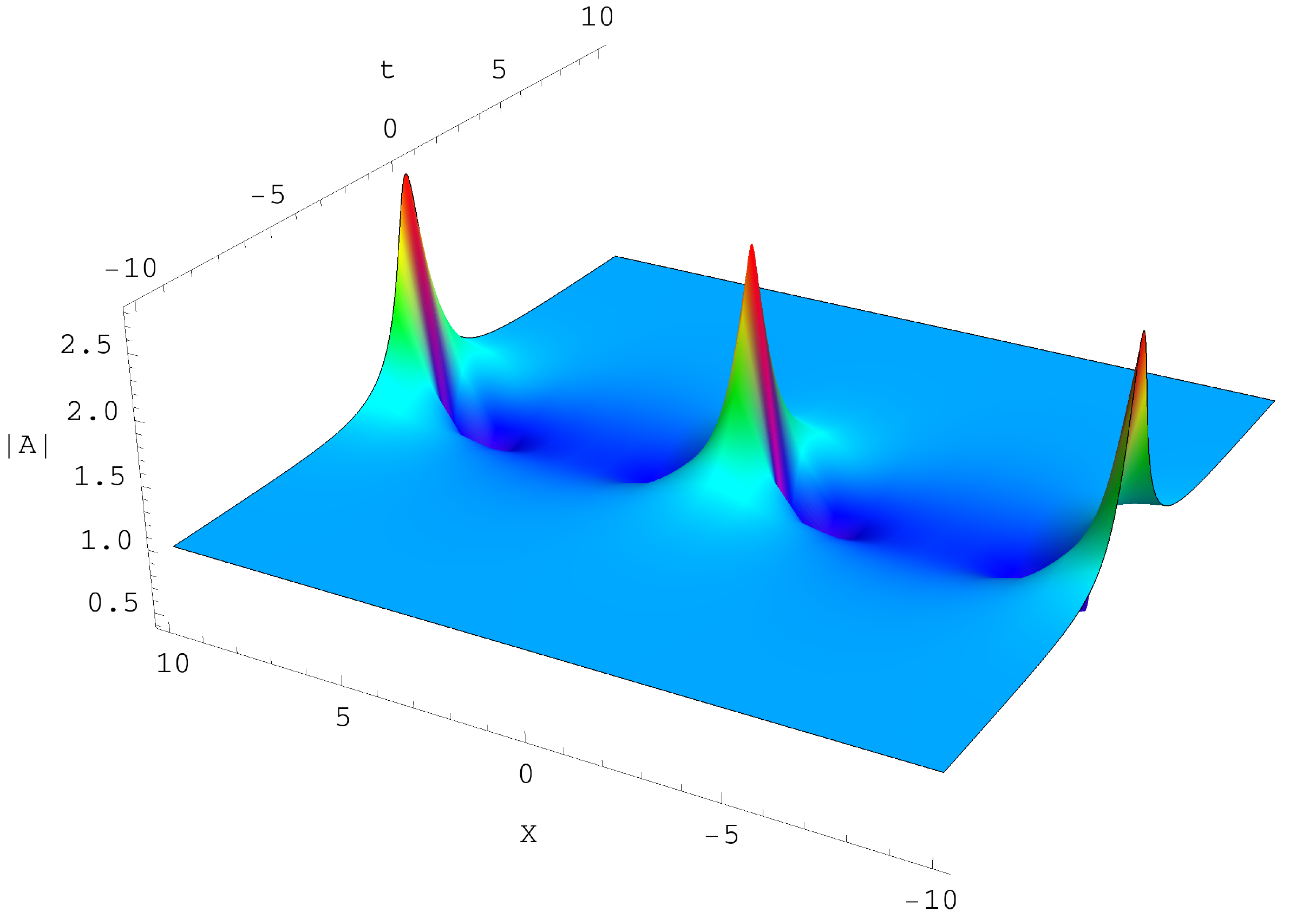}
             % \caption{{\small $|A|$ with $a=0.4$.}}
              \captionof{figure}{\small $|A|$ with $a=0.4$.}\label{Fig:2} 
        %  \end{figure}
      \end{minipage}
\end{minipage}

\medskip

%which induce the following conserved identities for \eqref{mNLS}:
%\be\label{Mass_w}
%\int |w|^2 + 2\re w = \hbox{conserved},
%\ee
%and 
%\be\label{Energy_w}
% \int |\partial_x w|^2 - \frac12 \int (|w|^2 + 2\re w)^2 =\hbox{conserved}.
%\ee
%Note that \eqref{Mass_w} is not well-defined for solutions in $H^1$ only. The second quantity \eqref{Energy_w} is conserved as long as the solution $w(t)$ remains in $H^1(\R)$, as a standard density argument involving smooth and rapidly decaying solutions shows. However, unlike as in the defocusing case \cite{Zhidkov, Gallo, Gerard}, this energy does not give a priori control of the dynamics. A third conserved quantity comes from the momentum law:
%\be\label{Momentum_w}
%\ima \int \overline{w} \partial_x w = \hbox{conserved}.
%\ee
%We conclude that the fact that \eqref{mNLS} has no conservation of mass reveals that it is not clear whether or not there is a reasonable GWP result, except possibly for small data.
%\medskip
%
%Note that both \eqref{Energy_w} and \eqref{Mass_w} do not give control on the $H^1$ norm of the solution. Using now the relation
%\be\label{dMass}
%\begin{aligned}
%\frac12\frac{d}{dt} \int |w|^2   = &~ 2\int \re w \ima w +  \int |w|^2 \ima w\\
% \leq &~(1 + \|\ima w\|_{L^\infty}) \int |w|^2,
%\end{aligned}
%\ee
%we see in principle no good control of the growth of the $L^2$ norm in time.
%\end{rem}

\medskip

A particular feature of $A$ above is its nonzero boundary value at infinity in time and space. Indeed, $A$ converges, as $t\to\pm\infty$, to the Stokes wave $e^{it}$, also solution of \eqref{NLS}:
\be\label{Asymptotic}
\lim_{t\to\pm \infty} \|A(t,x) - e^{\pm i \theta}e^{it}\|_{H^1_\sharp}=0, \quad e^{i\theta}=1-\al^2 -i\beta .
\ee
Here, $H^s_\sharp:= H^s_\sharp((0,\frac{2\pi}{a}))$ denotes the standard Sobolev space $H^s$ of $\frac{2\pi}{a}$-space periodic functions. Consequently, $A(t,x)$ exemplifies the \emph{modulational instability} phenomenon, which -roughly speaking- says that small perturbations of the Stokes wave are unstable and grow quickly. This unstable growth leads to a nontrivial competition with the (focusing) nonlinearity, time at which the solution is apparently stabilized. The Akhmediev breather is also a candidate  to explain the famous \emph{rogue waves}. An alternative explanation to the rogue wave phenomena is given by the notion of \emph{dispersive blow-up}, see Bona and Saut \cite{Bona_Saut}.
%), in which the solution, even if it is well-defined in $L^2$-based Sobolev spaces, has $L^\infty$ norm becoming unbounded in finite time. This argument works for linear and nonlinear equations as well. Therefore, a necessary condition for getting dispersive blow-up is an $H^s$ regularity where $s<\frac d2$, $d$ being the dimension of space, which does not fit our LWP assumptions. For further improvements and applications of these ideas to other dispersive models, see the work by Bona et al. \cite{BPSS}.  

\medskip

Two standard conserved quantities for \eqref{NLS} in the periodic setting are
\be\label{Mass}
M[u]: =\int_{0}^{\frac{2\pi}{a}} (|u|^2 -1), \qquad \hbox{(Mass)}
\ee
and 
\be\label{Energy}
E[u]:= \int_{0}^{\frac{2\pi}{a}}  \left( |u_x|^2 - \frac12   (|u|^2-1)^2 \right), \qquad \hbox{(Energy).}
\ee
A third one is given by \cite{AFM}
\be\label{F_NLS}
F[u]:= \int_{0}^{\frac{2\pi}{a}} \Big(|u_{xx}|^2 -3 (|u|^2-1)|u_x|^2 -\frac12((|u|^2)_x)^2  + \frac12 (|u|^2-1)^3 \Big). %\qquad \hbox{(Second energy)}
%F[u]:= \int \Big(|u_{xx}|^2 -3 |u|^2|u_x|^2 -2(\re (\bar u u_x))^2  + \frac12 |u|^2 (|u|^2-1)^4 \Big), \qquad \hbox{(Second energy)}
\ee
This third conserved quantity appears from the integrability of the equation.

\medskip

In this paper, we continue the work started by one of us in \cite{Munoz1}, where we proved that the Peregrine \cite{Peregrine,Akhmediev2} and Kuznetsov-Ma \cite{Kuznetsov,Ma} breathers are unstable under finite energy perturbations in any Sobolev space $H^s(\R)$, $s>\frac12$. Previously, the Peregrine soliton was showed to be numerically unstable under small perturbations by Klein and Haragus \cite{KH}. See \cite{Munoz1} for more details on those breathers, as well as a more or less accurate account of the current literature. 

\medskip

However, the stability analysis of \eqref{Ak} was left open because of its spatial periodic behavior. Our first and main result is the following:

\begin{thm}\label{Insta_A}
The Akhmediev breather \eqref{Ak} is unstable under small perturbations in $H^s_\sharp$, $s>\frac12$.
\end{thm}

By stability, we mean the following concept \cite{Munoz1}. Fix $s>\frac12$, and $t_0\in \R$. We say that a particular $\frac{2\pi}{a}$-periodic globally defined solution $U= e^{it}(1+W)$ of \eqref{NLS} is \emph{orbitally stable} in $H^s_\sharp(\frac{2\pi}{a})$ if there are constants $C_0,\ve_0>0$ such that, for any $0<\ve<\ve_0$,  
\be\label{Stability}
\begin{aligned}
\| u_0 - & ~{} U(t_0)\|_{H^s_\sharp} <  \ve\\
& \Downarrow  \\
\exists ~ x_0(t),\ga_0(t)\in \R ~  \hbox{such that }~ &  \sup_{t\in\R}  \|u(t) - e^{i\ga(t)}U(t,x-x_0(t)) \|_{H^s_\sharp} <C_0\, \ve.
\end{aligned}
\ee

\medskip
\noindent
Here $u(t)$ is the solution to the IVP \eqref{NLS} with initial datum $u(t_0)=u_0$ (see Proposition \ref{MT2}), and $x_0(t),\ga_0(t)$ can be assumed continuous because the IVP is well-posed in a continuous-in-time Sobolev space. If \eqref{Stability} is not satisfied, we will say that $U$ is {\bf unstable}. Note additionally that condition \eqref{Stability} requires $w$  globally defined, otherwise $U$ is trivially unstable, since $U$ is globally defined.

\medskip

The proof of Theorem \ref{Insta_A} uses \eqref{Asymptotic} in a crucial way: a modified Stokes wave is an attractor of the dynamics around the Akhmediev breather for large time. See also \cite{AMP1,AMP2} for
numerical studies of the stability of mKdV and Sine-Gordon breathers in the periodic and nonperiodic settings. 
Other rigorous stability results for breathers can be found in \cite{AM1,AM2, MP,Munoz,Alejo}. 

\medskip

No NLS \eqref{NLS} breather seems to be stable. In fact, Peregrine, Kuznetsov-Ma and now Akhmediev were shown to be unstable. 
This is not necessarily consequence of the nonzero background. Indeed, even breathers on zero background \cite{AFM}, 
called Satsuma-Yajima breathers, are unstable.

\medskip

Being $A$ unstable, it does not mean that it has no structure at all. In this paper we advance, 
following the ideas introduced in \cite{AFM}, that indeed, $A$ has a very rich (unstable) variational structure. In particular,

\begin{thm}\label{MT}
The Akhmediev breather $A$ \eqref{Ak} is a critical point of the functional
\[
\mathcal H[u]:= F[u] - \al^2 E[u],
\]
that is to say,  $\mathcal H'[A][w]=0$ for all $w\in H^2_\sharp$. In particular, for each $t\in\R$ $A(t,x)$ satisfies the nonlinear ODE
\be\label{Ec_A}
\begin{aligned}
& A_{(4x)} + 3A_x^2 \bar A +(4 |A|^2-3) A_{xx}+ A^2 \bar A_{xx}  + 2  |A_x|^2 A  \\
&\qquad  + \frac32 (|A|^2-1)^2 A + \alpha^2(A_{xx} +  (|A|^2-1) A) =0.
\end{aligned}
\ee
\end{thm}

The proof of this result follows easily from the methods in \cite{AFM}, in which one expands $\mathcal H[A+w]$. We get 
\[
\mathcal H[A+w]  =\mathcal H[A] + \mathcal H'[A][w] + O(\|w\|_{H^2_\sharp}^2).
\] 
Then, performing some lengthy computations, one proves that $\mathcal H'[A][w] =0$ independently of $w$. See Section \ref{3} for the proof.

% We give now the details
% of the proof, based in the ones performed in \cite{AFM}. 
%\end{proof}

\medskip

We believe that the variational structure appearing in breather solutions is independent of the well-posed character of the equation. In particular, we claim that the explicit breather of the strongly ill-posed bad Boussinesq equation 
\be\label{Bad}
u_{tt} - u_{xx} - \left(  u_{xx} + \frac{3}{2} u^2 \right)_{xx}=0.
\ee
has an associated rich variational structure \cite{ACM}. 
%
%\subsection{About the literature}  

%\medskip 
%
%
%
%\noindent
%{\bf Organization of this paper.} This paper is organized as follows. In Section \ref{2} we prove Theorem \ref{Insta_A}. Section \ref{3} is devoted to the proof of Theorem \ref{MT}.

\medskip

\noindent
{\bf Acknowledgments.} We thank the referee for his/her fruitful comments and suggestions which helped to improve this paper.

\medskip 

\section{Proof of Theorem \ref{Insta_A}}\label{2}

\medskip

The proof is not difficult at all. We just need a preliminary well-posedness result. Set 
\be\label{deco}
u(t,x) = A(t,x) + w(t,x), \quad w~ \hbox{ unknown}.
\ee
Then \eqref{NLS} becomes a modified NLS equation% with a zeroth order term, which is real-valued, and has the wrong sign:
\be\label{mNLS}
\begin{aligned}
i\partial_t w + \partial_x^2 w  &  =~ - G[w] , \\
 G[w]:= & ~ {} |A + w|^2(A+w) -|A|^2 A .
\end{aligned}
\ee

\begin{prop}\label{MT2}
The NLS equation \eqref{NLS} is locally well-posed for any initial data at time $t=t_0$ of the form $A(t_0,x) + w_0(x)$, with $w_0\in H^s_\sharp$, $s>\frac12$.
\end{prop}

\begin{proof}
See Appendix \ref{A}.
\end{proof}

Note that there is always a local solution $u$ of \eqref{NLS} such that $u(t)= A(t) +w(t)$, with  $w\in H^s_\sharp$. In particular, given time dependent parameters $x_0(t),\ga_0(t)\in\R$, if the decomposition $u(t)= e^{i\ga_0(t)}A(t,x-x_0(t)) +\tilde w(t)$ holds, then $\tilde w(t)$ still belongs to $H^s_\sharp$. This is not true in the non periodic case, see \cite{Munoz1}.

\medskip

We did not try to improve the local well-posedness result for \eqref{mNLS} because the flow contains a non oscillatory bad component in the case of small frequencies, see \cite{Munoz1} for details. In particular, Strichartz estimates are not available in this case. Also, the global well-posedness of \eqref{mNLS} is an open question.

\medskip

\subsection{End of proof} We only treat the case $t\to+\infty$, the other being very similar.  Fix $s>\frac12$. Let us assume that the Akhmediev breather $A$ in \eqref{Ak} is orbitally stable, as in \eqref{Stability}. Write (see \eqref{Asymptotic})
\be\label{Q_def}
\begin{aligned}
A(t,x)= &~ {}e^{it} (e^{i\theta} + Q(t,x)),\\
 Q(t,x):= &~{} \frac{\alpha^2 \cosh( \beta t) +i\beta \sinh(\beta t) }{ \sqrt{2a} \cos(\alpha x) -\cosh(\beta t)} +(\al^2+i\beta) .
\end{aligned}
\ee
Now consider, as a perturbation of the Akhmediev breather, the $\frac{2\pi}{a}$-periodic \emph{Stokes wave $ e^{i\theta}e^{it}$}.  Indeed, we have (see \eqref{Q_def}),
\[
\lim_{t\to +\infty} \|  A(t) -e^{i\theta}e^{it} \|_{H^s_\sharp}= \lim_{t\to +\infty} \| Q(t) \|_{H^s_\sharp} =0.
\]
Indeed, this follows from the identity
\be\label{Q_new}
Q(t,x)= \alpha^2 \left(1 - \frac{1}{ 1-\sqrt{2a} \frac{\cos(\alpha x)}{\cosh(\beta t)}} \right) +i\beta \left( 1- \frac{\tanh(\beta t) }{ 1-\sqrt{2a} \frac{\cos(\alpha x)}{\cosh(\beta t)}} \right). 
\ee
Therefore, we have two solutions to \eqref{NLS} that converge to the same profile as $t\to +\infty$. This fact contradicts the orbital stability, since for $x_0(t),\ga_0(t)\in \R$ given in \eqref{Stability}, and the definition of $A$ in \eqref{Ak},
\[
\begin{aligned}
 \| e^{i\theta} - e^{i\ga_0(0)} A(0,x-x_0(0)) \|_{H^s_\sharp} = &~{} \norm{ e^{i\theta} - e^{i\ga_0(0)} \Bigg[  1+ \frac{\alpha^2  }{ \sqrt{2a} \cos(\alpha (x-x_0(0)) -1}\Bigg]   }_{H^s_\sharp} \\
 = &~{}c_s>0,
\end{aligned}
\]
is a fixed number, but if $t_0=T$ is taken large enough, $ \| Q(T) \|_{H^s_\sharp}$ can be made arbitrarily small. Indeed, by classical interpolation ($\|u\|_{H^s_\sharp}^2:= \sum_{n\geq0} n^{2s}|\hat u(n)|^2$, and  $n^{2s} = n^{2(0+s.1)}$ and H\"older)
\be\label{interpolation}
\|Q(T)\|_{H^s_\sharp} \lesssim_s \|Q(T)\|_{L^2_\sharp}^{1-s}\|Q(T)\|_{H^1_\sharp}^{s},\quad s\in (0,1).
\ee
Now, to evaluate $\lim_{t\to +\infty}\|Q(t)\|_{L^2_\sharp}$ requires some care. Clearly from \eqref{Q_new} we have $Q(t,x)\to 0$ as $t\to +\infty$, for all $x\in [0,\frac{2\pi}{\alpha})$. Also,
\[
|Q(t,x)| \lesssim \frac{\alpha^2\sqrt{2a}}{(1-\sqrt{2a})\cosh(\beta t)} + \frac{\beta}{(1-\sqrt{2a})}\left((1-\tanh(\beta t))  + \frac{\sqrt{2a}}{\cosh(\beta t)}\right).
\]
Therefore, by using dominated convergence we conclude. As for the derivative, note that 
\be\label{Q_new_x}
\partial_x Q(t,x)= \frac{\alpha^3 \sqrt{2a} \sin(\al x)}{\cosh(\beta t) \left( 1-\sqrt{2a} \frac{\cos(\alpha x)}{\cosh(\beta t)} \right)^2 } +i\frac{ \alpha \beta\sqrt{2a} \tanh(\beta t) \sin(\al x)}{\left( 1-\sqrt{2a} \frac{\cos(\alpha x)}{\cosh(\beta t)} \right)^2}. 
\ee
Proceeding in a similar fashion as in the $L^2$ norm, we have $\lim_{t\to +\infty}\|\partial_x Q(t)\|_{L^2_\sharp}=0.$ Therefore, we conclude from \eqref{interpolation} that $ \| Q(T) \|_{H^s_\sharp}$ can be made arbitrarily small if $T$ is sufficiently large.

\medskip

Note finally that the Cauchy problem for \eqref{NLS} with initial data at time $T$ given by $u_0=e^{iT}e^{i\theta} = A(T)-e^{iT}Q(T)$ is well-defined from Proposition \ref{MT2}, since $e^{iT}Q(T) \in H^s_\sharp$. This proves Theorem \ref{Insta_A}.

\begin{rem}
We conjecture that any soliton solution constructed using B\"acklund transformations, with attached Akhmediev breathers, must be unstable.
\end{rem}

\medskip

\section{Proof of Theorem \ref{MT}}\label{3}

Explicitly, we have from \eqref{Energy} and \eqref{F_NLS}, integration by parts, and the periodic character of $A$ and its spatial derivatives at the boundaries, and $w$ its first and second spatial derivatives,
\[%be\label{ExpandH0}
\begin{aligned}
&\mathcal H[A+w]= F[A+w] - \al^2 E[A+w]\\
&\quad =\int_{0}^{\frac{2\pi}{a}}\!\! \Big(|A_{xx} + w_{xx}|^2 -3 (|A+w|^2-1)|A_x+w_x|^2 -\frac12((|A+w|^2)_x)^2  + \frac12 (|A+w|^2-1)^3 \Big)\\
&\quad \quad  -\alpha^2\int_{0}^{\frac{2\pi}{a}}  \left( |A_x+w_x|^2 - \frac12   (|A+w|^2-1)^2 \right)\\
&\quad = \mathcal H[A] + \int_{0}^{\frac{2\pi}{a}} \Big(2\re(A_{4x}\bar{w}) - 3(|A|^2-1)2\re(A_x\bar{w}_x) - 3(2\re(A\bar{w}))|A_x|^2\\ 
&\quad \quad \qquad \qquad \qquad - (|A|^2)_x(2\re(A\bar{w}))_x + \frac32(|A|^2-1)^22\re(A\bar{w})\Big)\\
& \quad \quad -\alpha^2\int_{0}^{\frac{2\pi}{a}}  \Big(-2\re(A_{xx}\bar{w}) -(|A|^2-1)2\re(A\bar{w}) \Big) + O(\|w\|_{H^2_\sharp}^2)\\
&\quad = \mathcal H[A] + 2\re\int_{0}^{\frac{2\pi}{a}} \Big(A_{4x}\bar{w} - 3(|A|^2-1)A_x\bar{w}_x - 3A\bar{w}|A_x|^2 + (|A|^2)_{xx}A\bar{w}\\
 &\quad \quad  \qquad \qquad \qquad\qquad  + \frac32(|A|^2-1)^2A\bar{w} -\alpha^2\Big[-A_{xx}\bar{w} -(|A|^2-1)A\bar{w} \Big]\Big) + O(\|w\|_{H^2_\sharp}^2)\\
&\quad = \mathcal H[A] \\
&\quad \quad + 2\re\int_{0}^{\frac{2\pi}{a}}\bar{w} \Big(A_{4x} + 3(|A|^2-1)A_{xx} + 3(A_x^2\bar{A}+A|A_{x}|^2) - 3A|A_x|^2\\
&\quad \quad  \qquad \qquad \qquad + A_{xx}|A|^2 + A^2\bar{A}_{xx} + 2A|A_x|^2 + \frac32(|A|^2-1)^2A +\alpha^2(A_{xx} + (|A|^2-1)A)\Big) \\
&~ \quad \quad + O(\|w\|_{H^2_\sharp}^2)\\ 
&\quad = \mathcal H[A] +2\re\int_{0}^{\frac{2\pi}{a}}\bar{w} \Big(A_{(4x)} + 3A_x^2 \bar A +(4 |A|^2-3) A_{xx}+ A^2 \bar A_{xx}  + 2  |A_x|^2 A  \\
&\quad \quad  \qquad \qquad \qquad\qquad  + \frac32 (|A|^2-1)^2 A + \alpha^2(A_{xx} +  (|A|^2-1) A)\Big)  + O(\|w\|_{H^2_\sharp}^2),
 \end{aligned}
\]
and therefore we get 
\[
\mathcal H[A+w]  =\mathcal H[A] + \mathcal H'[A][w] + O(\|w\|_{H^2_\sharp}^2).
\] 
Then, performing some lengthy computations (see Appendix \ref{A0}) one proves that $\mathcal H'[A][w] =0$ independently of $w$.   This proves Theorem \ref{MT}.

% We give now the details
% of the proof, based in the ones performed in \cite{AFM}. 

% A(t,x):=& ~ e^{it}\Bigg[  1+ \frac{\alpha^2 \cosh( \beta t) +i\beta \sinh(\beta t) }{ \sqrt{2a} \cos(\alpha x) -\cosh(\beta t)}\Bigg], \qquad t,x\in\R,\\
%  \beta=&~ (8a(1-2a))^{1/2}, \quad \alpha=(2(1-2a))^{1/2},

\bigskip

\appendix

\section{Proof of \eqref{Ec_A}}\label{A0}
%KMiden
Following \cite{AFM}, let us use the notation for the Akhmediev breather solution \eqref{Ak}:
\be\label{AppA}
\begin{aligned}
&A=e^{it}\left(1 + \frac{M}{N}\right), \quad\text{with}\\
& M:=\alpha^2 \cosh( \beta t) +i\beta \sinh(\beta t) ,\\
& N:= \sqrt{2a} \cos(\alpha x) -\cosh(\beta t).
\end{aligned}
\ee
Now, we rewrite the identity \eqref{Ec_A} in terms of $M,N$ in the following way
\be\label{AppKAAA}
\eqref{Ec_A} = \frac{e^{it}}{N^5}\sum_{i=1}^{6}R_i,
\ee
\noindent
with $R_i$ given explicitly by:
\be\label{AppAs1}
\begin{aligned}
R_1:= & ~{} \frac{1}{2}N\Big(6iMN_tN_x^2-2iN(N_x(M_tN_x+M(iN_x+2N_{xt})) + N_t(2M_xN_x+MN_{xx}))\\
& ~ {} \qquad  +N^3(M_{xx}-iM_{xxt})+N^2(-2M_x(N_x-iN_{xt}) \\
& ~ {} \qquad +i(2N_xM_{xt}+N_tM_{xx}+iMN_{xx}+M_tN_{xx}+MN_{xxt}))\Big),
\end{aligned}
\ee
\be\label{AppAs2}
\begin{aligned}
R_2:=&~{}\frac12\Big(2M(\bar{M} + N) + N(2\bar{M} + (\alpha^2-1)N)\Big)\\
&\cdot\Big(2MN_x^2 + N^2M_{xx}-N(2M_xN_x+MN_{xx})\Big),
%-(\bar{M}+N)(NM_x-MN_x)^2,
\end{aligned}
\ee
\be\label{AppAs3}
\begin{aligned}
&R_3:=~{}(M+N)(-NM_x+MN_x)(N\bar{M}_x - \bar{M}N_x),
\end{aligned}
\ee
\be\label{AppAs4}
\begin{aligned}
&R_4:=~{}\frac12(\bar{M}+N)(NM_x-MN_x)^2,
\end{aligned}
\ee
\be\label{AppAs5}
\begin{aligned}
&R_5:=~{}\frac{1}{2}N^2(M+N)\Big((\frac32 - \alpha^2)N^2 + (-3 + \alpha^2)(\bar{M}+N)(M+N)\Big),
\end{aligned}
\ee
and
\be\label{AppAs6}
\begin{aligned}
R_6:=&~{} \frac34(M+N)^3(\bar{M}+N)^2.
%(\bar{M}N+M(-\bar{M}+N))^2,
\end{aligned}
\ee
Now substituting the explicit functions $M,N$ \eqref{AppA} in $R_i,~i=1,\dots,6$ and collecting terms, we get 
%{Cos[t \[Alpha]] -> X, Sin[t \[Alpha]] -> Y, Cosh[x \[Beta]] -> R,  Sinh[x \[Beta]] -> S}
\[%be\label{AppKMsum}
\begin{aligned}
&\sum_{i=1}^{6}R_i \\
&~{} = a_1\cosh(t\beta) + a_2\cosh^3(\beta t) + a_3\cosh^5(\beta t) + a_4\sinh(t\beta) + a_5\cosh^2(\beta t)\sinh(\beta t) \\
&\quad +  a_6\cosh^4(\beta t)\sinh(\beta t) + a_7\cos( \alpha x) +  a_8\cosh^2(\bt t)\cos(\al t)  + a_9\cosh^4(\bt t)\cos(\al x)\\
&\quad + a_{10} \cosh(\bt t)\sinh(\bt t)\cos(\al x)  + a_{11}\cosh^3(\bt t)\sinh(\bt t)\cos(\al x)   +  a_{12}\cosh(\bt t)\cos^2(\al x) \\
& \quad +  a_{13}\cosh^3(\bt t)\cos^2(\al x) +  a_{14}\cosh^2(\bt t)\cos^3(\al x) +  a_{15}\cosh(\bt t)\cos^4(\al x),                                                                                                                                                                                                                                                                                                                                                                                                            
\end{aligned}
\]
with coefficients $a_i,~i=1,\dots,15$ given as follows
\[
\begin{aligned}
%&\quad \text{REESCRIBIR ESTAS FUNCIONES, PONIENDO UNAS EN FUNCION DE LAS OTRAS PARA AHORRAR ESCRITURA}\\
&a_1=\frac32 (-1 + \al^2) \bt^2 (-4 a \al^2 + \bt^2),\\
&a_2=(-(-1 + \al^2) \bt^2 (-5 \al^2 + 3 \al^4 + 3 \bt^2) + 2 a (-5 \al^6 + 3 \al^8 - \al^2 \bt^2 + 3 \al^4 \bt^2)) ,\\%R^3
&a_3=\frac12(-1 + \al^2) (-10 \al^6 + 3 \al^8 - 10 \al^2 \bt^2 + 3 \bt^4 +  \al^4 (8 + 6 \bt^2)),\\% R^5
&a_4=\frac32 i \bt^3 (-4 a \al^2 + \bt^2),\\% S
&a_5=i \bt (\bt^2 (5 \al^2 - 3 \al^4 - 3 \bt^2) +  a (-8 \al^4 + 6 \al^6 + 6 \al^2 \bt^2)),\\% R^2 S
&a_6=\frac12i \bt (-10 \al^6 + 3 \al^8 - 10 \al^2 \bt^2 + 3 \bt^4 +    \al^4 (8 + 6 \bt^2)),\\% R^4 S
&a_7=\frac32 \sqrt{2a}\bt^2 (-4 a \al^2 + \bt^2),\\% X
&a_8=-\sqrt{2a}(\bt^2 (-7 \al^2 + 5 \al^4 + 3 \bt^2) +  a (-6 \al^6 + 2 \al^2 \bt^2)),\\% R^2 X
&a_9=\frac12\sqrt{2a}(-24 \al^6 + 7 \al^8 - 16 \al^2 \bt^2 + 3 \bt^4 +  10 \al^4 (2 + \bt^2)),\\% R^4 X
&a_{10}=2 i \sqrt{2a}\al^2 \bt (4 a \al^2 - \bt^2),\\% R S X
&a_{11}=2 i \sqrt{2a}\al^2 \bt (-2 \al^2 + \al^4 + \bt^2),\\% R^3 S X
&a_{12}=4 a \al^2 (-\bt^2 + 2 a (\al^4 + \bt^2)),\\% R X^2
&a_{13}=6 a \al^2 (-2 \al^2 + \al^4 + \bt^2),\\% R^3 X^2
&a_{14}=2 \sqrt{2a}a\al^2 (-2 \al^2 + \al^4 + \bt^2),\\% R^2 X^3
&a_{15}=-4 a^2 \al^2 (-2 \al^2 + \al^4 + \bt^2).% R X^4
\end{aligned}
\]
Finally, using that $\al =\sqrt{2 (1-2 a)}$ and $\bt =\sqrt{8a (1-2 a)}$, we have that all $a_i$ vanish, and we conclude.

\section{Sketch of Proof of Proposition \ref{MT2}}\label{A}

\medskip

First of all, we have from \eqref{mNLS} that 
\[
G[w]=  2A\re(A \bar w)  + |A|^2 w + A |w|^2  + 2\re(A \bar w)w +|w|^2w.
\]
By scaling and the subcritical character of \eqref{mNLS}, we can assume that the linear term in $G[w]$ above is small. We can also assume the initial time $t_0=0$. By the Duhamel's formula, we have
\[
w(t) = e^{it\partial_x^2} w_0 - \int_0^t e^{i(t-s)\partial_x^2}G[w](s)ds.
\]
Hence, applying the standard Sobolev estimates in $H^s_\sharp$, with $s>\frac12$, we readily obtain the contraction principle required. Note that no use of Strichartz estimates is needed. See \cite{Cazenave} or \cite{LP} for additional details on the fixed point argument. We skip the details.

\medskip

\subsection*{Conflict of interest statement:} The authors certify that no conflict of interest, of any possible type, is affected to this article. 

\medskip

\providecommand{\bysame}{\leavevmode\hbox to3em{\hrulefill}\thinspace}
\providecommand{\MR}{\relax\ifhmode\unskip\space\fi MR }
\providecommand{\MRhref}[2]{%
  \href{http://www.ams.org/mathscinet-getitem?mr=#1}{#2}
}
\providecommand{\href}[2]{#2}

\end{document}